\documentclass[12pt, reqno]{amsart}

\usepackage{amsmath,amscd, float,rotating}
\usepackage{pb-diagram}
\usepackage{latexsym}
\usepackage{amsfonts,color}
\usepackage{amsmath}
\usepackage{amssymb}
\usepackage{txfonts}

\newcommand{\vol}{\mathrm{Vol}}
\newcommand{\diam}{\mathrm{diam}}

\newcommand{\Ric}{\mathrm{Ric}}
\newcommand{\Rm}{\mathrm{Rm}}

\newcommand{\de}{\partial}

\newcommand{\ve}{\varepsilon}

\newcommand{\ti}[1]{\tilde{#1}}
\newcommand{\wt}{\widetilde}

\renewcommand{\leq}{\leqslant}
\renewcommand{\geq}{\geqslant}
\renewcommand{\le}{\leqslant}
\renewcommand{\ge}{\geqslant}

\newtheorem{theorem}{Theorem}[section]

\newtheorem{corollary}[theorem]{Corollary}
\newtheorem{proposition}[theorem]{Proposition}

\numberwithin{equation}{section}

\theoremstyle{definition}
\newtheorem{remark}[theorem]{Remark}

\theoremstyle{definition}

\numberwithin{equation}{section}

\begin{document}

\title[$\epsilon$-regularity for Ricci flow]{Global $\ve$-regularity for 4-dimensional Ricci flow with integral scalar curvature bound}

\author{Wangjian Jian}

\address{Hua Loo-Keng Center of Mathematical Sciences, Academy of Mathematics and Systems Science, Chinese Academy of Sciences, Beijing, 100190, China.}

\email{wangjian@amss.ac.cn}

\pagestyle{headings}

\begin{abstract}
Ge-Jiang (Geom Funct Anal 27:1231-1256, 2017) proved global $\ve$-regularity for 4-dimensional Ricci flow with bounded scalar curvature. In this note, we extend this result to 4-dimensional Ricci flow with integral bound on the scalar curvature.
\end{abstract}

\maketitle

\section{Introduction}
\label{section:introduction}

In \cite{CT06}, Cheeger-Tian proved $\ve$-regularity theorem for $4$-dimensional Einstein manifolds without volume assumption. If one assume that the volume is noncollapsed, such regularity estimate is well known by  \cite{An89,BKN89,Ti90}. If one assume average $L^2$ bound on $\Rm$, such an $\ve$-regularity was proved by Anderson \cite{An92}, see also \cite{TiVi,TiVi2,Car14} for a proof of the $\ve$-regularity under such condition.

Cheeger-Tian \cite{CT06} conjectured that a similar $\ve$-regularity should hold on higher dimensional Einstein manifolds, $4$-dimensional shrinking solitons, Ricci flows and critical metrics (which are  metrics satisfying $\Delta \Ric=\Rm\ast \Ric$).

In \cite{Li10}, Li proved a local smoothing result for Riemannian manifolds with bounded Ricci curvatures in dimension 4. In \cite{MuWa15}, Munteanu-Wang proved that $4$-dimensional shrinking Ricci solitons have bounded Riemann curvature provided a bounded scalar curvature condition, where the Riemann curvature bound depends on the local geometry around the base point.  In \cite{Hu20}, Huang proved an $\ve$-regularity for shrinking Ricci solitons with $\ve$ depending on the distance to the base point.

In \cite{GJ17}, Ge-Jiang proved an $\ve$-regularity for shrinking Ricci solitons which generalizes Cheeger-Tian's result. Moreover, by proving a Backward Pseudolocality estimate for Riemann curvature, Ge-Jiang proved $\ve$-regularity for Ricci flow with bounded scalar curvature, which partially confirms Cheeger-Tian's conjecture in the Ricci flow case. 

In this short note, we extend Ge-Jiang's (\cite{GJ17}) $\ve$-regularity for Ricci flow to the case of integral bound on scalar curvature. The following is our main result.

\begin{theorem}\label{main theorem}
Given $p_0>2$, there exist constants $\ve=\ve(p_0)$, $C=C(p_0)$ such that the following holds.

Let $(M, g(t)_{t\in(-1,1)})$ be a Ricci flow on a compact 4-dimensional manifold $M$ with $\left\|R(t)\right\|_{L^{p_0}(M,g(t))}\leq S$ for all $t\in(-1,1)$, and suppose that at time $0$ we have
\begin{equation}\label{l2 bound on Rm}
\int_{M}\left|\Rm(g(0))\right|^2dg(0)\leq \ve,
\end{equation}
and the diameter $\diam(M, g(0))=D$. Then the curvature at $t=0$ is bounded:
\begin{equation}\label{bound on Rm}
\sup_{M}\left|\Rm(g(0))\right|\leq C\max\left\{D^{-2}, 1\right\}S^{\frac{p_0}{p_0-2}}.
\end{equation}
\end{theorem}

\begin{remark}
By scaling the given Ricci flow in Theorem \ref{main theorem} parabolically by $Q:=S^{\frac{p_0}{p_0-2}}$, we may assume that $S=1$ in Theorem \ref{main theorem}. Also, to prove Theorem \ref{main theorem}, we may assume $D\geq 1$ since we can rescale the metric parabolically by $D^{-2}$ when $D\leq 1$.
\end{remark}

{\bf Acknowledgements.} The author would like to thank his advisor Gang Tian for constant encouragement and support.  The author also would like to thank Wenshuai Jiang, Bin Guo and Q.S. Zhang for many helpful discussions.

\section{Curvature Derivative Estimates}
\label{Curvature Derivative Estimates}
In this section, we will prove the curvature derivative estimate, which is the key estimate toward our $\ve$-regularity estimate.

First, we prove the following improved curvature estimate on Ricci flow. The same result for Ricci flow with bounded scalar curvature could be found in Bamler-Zhang \cite[Lemma 6.1]{BZ17}; see also Wang \cite{Wang12}.

\begin{proposition}[Improved Curvature Estimate]\label{Prop: improved curvature estimate}
Given $n\geq 4$, $p_0>\frac{n}{2}$, for $m=0,1,2,\dots$, there exist constants $C_m=C_m(n,p_0)<\infty$ such that the following holds.

Let $(M^n,g(t)_{t\in[0,T)})$ be a Ricci flow on a compact $n$-dimensional manifold $M$ with $\left\|R(t)\right\|_{L^{p_0}(M,g(t))}\leq\rho\leq 1$ for all $t\in[0,T)$. Let $x_0\in M$, $t_0\in[0,T)$ and $0< r_0^2\leq \min\left\{1, t_0\right\}$. Assume that the ball $B(x_0, t_0, r_0)$ is relatively compact, and we have $\left|\Rm\right|\leq r_0^{-2}$ on $P(x_0, t_0, r_0, -r_0^{-2})$.

Then we have $\left|\nabla^m\Ric\right|(x_0, t_0)\leq C_m\rho^{\frac{1}{2}} r_0^{-m-1-\frac{n}{2p_0}}$ for $m=0,1,2,\dots$, and $\left|\de_t\Rm\right|(x_0, t_0)\leq C_0\rho^{\frac{1}{2}} r_0^{-3-\frac{n}{2p_0}}$.
\end{proposition}
\begin{proof}
We follow the proof of Bamler-Zhang \cite[Lemma 6.1]{BZ17}. We first rescale the given Ricci flow parabolically by $r_0^{-2}$ such that $r_0=1$. Then for the new flow, we have 
\begin{equation}\label{integral bound on R for rescaled flow}
\left\|R(t)\right\|_{L^{p_0}(M,g(t))}\leq\rho r_0^{2-\frac{n}{p_0}},  \qquad \text{for all} \qquad  t\in[0,r_0^{-2}T).
\end{equation}
and $\left|\Rm\right|\leq 1$ on $P(x_0, t_0, 1, -1)$.

Observe that by Shi's estimates (cf \cite{Sh89}) there are universal constants $D_0, D_1, \dots < \infty$ such that
\begin{equation} \label{eq: Shi estimate}
|{\nabla^m \Rm}| < D_m,  \qquad \text{on} \qquad P (x_0, t_0, \tfrac12, - \tfrac12).
\end{equation}
for all $m = 0, 1, 2 \dots$. Let $\alpha > 0$ be a small constant, whose value will be determined later, and consider the exponential map
$$ p : B(0,\alpha) \subset \mathbb{R}^n \longrightarrow B(x_0, t_0,\alpha ) \subset M. $$
based at $x_0$ with respect to the metric $g(t_0)$. Let $\ti{g}(t) = p^* g(t)$, $t \in [t_0 - \alpha^2, t_0]$ be the rescaled pull back of $g(t)$. By Jacobi field comparison and distance distortion estimates under the Ricci flow, we find that there is a universal choice for $\alpha$ such that the metric $\ti{g}(t)$ is 2-bilipschitz to the Euclidean metric on $\mathbb{R}^n$ for all $t \in [ t_0 - \alpha^2, t_0]$. Moreover, it follows from (\ref{eq: Shi estimate}) that there are universal constants $\wt{D}_0, \wt{D}_1, \ldots < \infty$ such that
$$ |{\de^m \ti{g}(t)}| < \wt{D}_m,  \qquad \text{on} \qquad B(0, \alpha ) \times [t_0 - \alpha^2 , t_0]. $$

Let now $\phi \in C_0^\infty (B(0, \alpha))$ be a cutoff function that satisfies $0 \leq \phi \leq 1$ everywhere and $\phi = 1$ on $B(0, \tfrac12 \alpha)$.
This cutoff function can be chosen such that $|\partial \phi|, |\partial^2 \phi|$ are bounded by some universal constant $C_1 < \infty$.
This implies that there is a universal constant $C_2 < \infty$ such that
$$ |{\Delta_{\ti{g}(t)} \phi}| < C_2, \qquad \text{for all} \qquad t \in [t_0 - \alpha^2, t_0]. $$
We now make use of the evolution equation for the scalar curvature of $\ti{g}(t)$:
$$ \de_t R =  \Delta_{\ti{g}(t)} R + 2 |{\Ric}|^2. $$
Integrating this equation against $\phi$ and using integration by parts yields $\lambda$
\begin{equation}\label{Equ: integration of evolution equation of R 1}
\begin{split}
&\bigg| \partial_t \int_{B(0, \alpha )} R(\cdot, t) \phi  d\ti{g}(t) - \int_{B(0, \alpha )} 2 |{\Ric}(\cdot, t)|^2 \phi d\ti{g}(t) \bigg|\\
=& \bigg|\int_{B(0, \alpha )} \Delta_{\ti{g}(t)}R(\cdot, t) \phi  d\ti{g}(t) - \int_{B(0, \alpha )} R(\cdot, t)^2 \phi d\ti{g}(t)\bigg|\\
\leq& \int_{B(0, \alpha )} |R(\cdot, t)| \cdot|\Delta_{\ti{g}(t)}\phi|  d\ti{g}(t) + \int_{B(0, \alpha )} |R(\cdot, t)|^2 \phi d\ti{g}(t)\\
\leq& C_2\int_{B(0, \alpha )} |R(\cdot, t)| d\ti{g}(t) + \int_{B(0, \alpha )} |R(\cdot, t)|^2 d\ti{g}(t)\\
\leq& C_2\left\|R(t)\right\|_{L^{p_0}(M,\ti{g}(t))}\cdot \vol_{\ti{g}(t)}(B(0, \alpha))^{\tfrac{p_0-1}{p_0}} + \left\|R(t)\right\|_{L^{p_0}(M,\ti{g}(t))}^2\cdot \vol_{\ti{g}(t)}(B(0, \alpha))^{\tfrac{p_0-2}{p_0}}\\
\leq& C(n, p_0)\cdot\rho r_0^{2-\frac{n}{p_0}}.\\
\end{split}
\end{equation}
where we have used the fact that $p_0>\frac{n}{2}\geq 2$ and $\rho\leq 1$. Integration in time and applying H\"older inequality as in Equation \eqref{Equ: integration of evolution equation of R 1} gives
\begin{equation}\label{Equ: integration of evolution equation of R 2}
\begin{split}
& \int_{t_0-\alpha^2}^{t_0}\int_{B(0, \alpha )} 2 |{\Ric}(\cdot, t)|^2 \phi d\ti{g}(t)dt \\
\leq& \int_{B(0, \alpha )} |R(\cdot, t_0)|^2 \phi d\ti{g}(t_0)+\int_{B(0, \alpha )} \left|R(\cdot, t_0-\alpha^2)\right|^2 \phi d\ti{g}(t_0-\alpha^2)+C(n, p_0)\cdot\rho r_0^{2-\frac{n}{p_0}}\\
\leq& C(n, p_0)\cdot\rho r_0^{2-\frac{n}{p_0}}.\\
\end{split}
\end{equation}
Hence we have 
$$\Vert {\Ric} \Vert_{L^2 \left(B(0, \frac12 \alpha ) \times [t_0 - \alpha^2, t_0]\right)} \leq C(n, p_0)\cdot\rho^{\frac{1}{2}} r_0^{1-\frac{n}{2p_0}}.$$
Note that $\Ric$ satisfies the linear parabolic evolution equation
\begin{equation} \label{Equ: Ric evolution}
 \left(\de_t - \Delta_{\ti{g}(t)} - 2 \Rm \right) \Ric = 0.
\end{equation}
The coefficients of this equation are universally bounded in every $C^m$-norm. Hence it follows from standard parabolic theory that for some universal $\ti{C} < \infty$
$$ |{\Ric}| (x_0, t_0)\leq \ti{C}\cdot 
\Vert {\Ric} \Vert_{L^2 \left(B(0, \frac12 \alpha ) \times [t_0 - \alpha^2, t_0]\right)} \leq C(n, p_0)\cdot\rho^{\frac{1}{2}} r_0^{1-\frac{n}{2p_0}}. $$
By applying the lemma at smaller scales, we obtain that
$$ |{\Ric}| < C(n, p_0)\cdot\rho^{\frac{1}{2}} r_0^{1-\frac{n}{2p_0}}, \qquad \text{on} \qquad B(0,\tfrac14 \alpha) \times \left[t_0 - \tfrac12 \alpha^2, t_0 \right]. $$
We can now apply the Schauder estimates on (\ref{Equ: Ric evolution}) and obtain that for all $m \geq 0$
$$ |\nabla^m {\Ric}|(x_0, t_0)\leq C_m(n, p_0)\cdot\rho^{\frac{1}{2}} r_0^{1-\frac{n}{2p_0}}. $$
for some constants $C_m(n, p_0)< \infty$.

For the bound on $|\de_t \Rm|(x_0, t_0)$ observe that at $(x_0, t_0)$ (compare also with \cite[Lemma 7.2]{Ha82})
$$ \partial_t \Rm_{abcd} = \Ric_{ai} \Rm_{ibcd} + \Ric_{bi} \Rm_{aicd} 
+ \nabla_{ac} \Ric_{bd} + \nabla_{ad} \Ric_{bc} + \nabla_{bc} \Ric_{ad} - \nabla_{bd} \Ric_{ac}.  $$
so that
$$ | {\partial_t \Rm} |(x_0, t_0) \leq C_8 |{\Ric (x_0, t_0)}| \cdot |{\Rm (x_0, t_0)}| + \left|{\nabla^2 \Ric (x_0,t_0)}\right|\leq C(n, p_0)\cdot\rho^{\frac{1}{2}} r_0^{1-\frac{n}{2p_0}}. $$
Rescaling back will give us the desired estimates. This finishes the proof of the Proposition.
\end{proof}

\begin{proposition}[Global Backward Pseudolocality]\label{Prop: global backward pseudolocality}
Given $n\geq 4$, $p_0>\frac{n}{2}$, there exist constant $\ve=\ve(n,p_0)>0$ such that the following holds.

Let $(M^n,g(t)_{t\in(-1,1)})$ be a Ricci flow on a compact $n$-dimensional manifold $M$ with $\left\|R(t)\right\|_{L^{p_0}(M,g(t))}\leq 1$ for all $t\in(-1,1)$. If we have $\sup_{M}\left|\Rm(g(0))\right|\leq A$ with $A\geq 1$. Then we have
\begin{equation}\label{Equ: global backward pseudolocality}
\sup_{M}\left|\Rm(g(t))\right|\leq \ve^{-2}A.
\end{equation}
for all $-\ve^2 A^{-1}\leq t\leq 0$.
\end{proposition}
\begin{proof}
Denote $f(t)=\sup_{M}\left|\Rm(g(t))\right|$. If we scaling the given Ricci flow parabolically by $A$, we can make that $A=1$ for the new flow, and we still have $\left\|R(t)\right\|_{L^{p_0}(M,g(t))}\leq A^{-1+\frac{n}{2p_0}}\leq 1$ for the new flow. We only need to prove the following Claim.

\textbf{Claim:} There exists $\ve(n, p_0)>0$ such that if $r\le \ve(n, p_0)$ and if $f(s)\le r^{-2}$ with fixed $s\in [-\frac{1}{2}+2r^2,0]$, then
$$\sup_{-r^2+s\le t\le s}f(t)\le 2r^{-2}.$$

We prove this by induction on $r=r_k=2^{-k}$. Since $M$ is compact, there exists $k_0$ depending on $(M,g(t))$ with $t\in [-3/4,0]$ such that $\sup_{-3/4\le t\le 0}\sup_{(M,g(t))}|\Rm|\le r_{k_0}^{-2}$. Thus the claim is true for such $r=r_{k_0}$.

Now we assume the claim holds for all $r\le r_i$ with $r_i\le \ve$ for some $\ve(n, p_0)$ to be determined. We are going to prove the claim for $r=r_{i-1}$ provided $r_{i-1}\le  \ve$.

Assume $f(s)\le r^{-2}_{i-1}$ for some $-\frac{1}{2}+2r_{i-1}^2\le s\le 0$. By induction for $f(s)\le r^{-2} _i$ we have
$$\sup_{-r_i^2+s\le t\le s}f(t)\le 2 r_i^{-2}.$$
By induction again for $s'=s-r_i^2$, we get
$$ \sup_{-r_{i+1}^2-r_i^2+s\le t\le s-r_i^2}f(t)=\sup_{-5r_{i+1}^2+s\le t\le s-r_i^2}f(t)\le 2 r_{i+1}^{-2}. $$
Applying our Improved Curvature Estimate Proposition \ref{Prop: improved curvature estimate}, we have derivative estimates
$$\sup_{-r_i^2+s\le t\le s}\Big|\partial_t|\Rm|\Big|\le C(n, p_0) r_{i+1}^{-3-\frac{n}{2p_0}}.$$
Hence for all $-r_i^2+s\le t\le s$, we have
\[
\begin{split}
f(t)\le& f(s)+(s-t)C(n, p_0)r_{i+1}^{-3-\frac{n}{2p_0}} \\
\leq& r_{i-1}^2+C(n, p_0)r_i^2r_{i+1}^{-3-\frac{n}{2p_0}}\\
\leq& 1+C(n, p_0)r_{i-1}^{-1-\frac{n}{2p_0}}\\
=& \left(1+C(n, p_0)r_{i-1}^{1-\frac{n}{2p_0}}\right)\cdot r_{i-1}^{-2}.\\
\end{split}
\]
Since $p_0>\frac{n}{2}$, we can choose $\ve=\ve(n, p_0)$ small enough such that $C(n, p_0)r_{i-1}^{1-\frac{n}{2p_0}}\leq \frac{1}{8}$, then we arrive at
$$\sup_{-r_i^2+s\le t\le s}f(t)\le \frac{9}{8}r_{i-1}^{-2}\le r_i^{-2}.$$
By using the same argument as above to $s\leftarrow s-r_i^2$, we can show that
$$\sup_{-2r_i^2+s\le t\le s-r_i^2}f(t)\le \frac{10}{8}r_{i-1}^{-2}\le r_i^{-2}.$$
Using the same argument twice to $s\leftarrow s-2r_i^2$ and $s\leftarrow s-3r_i^2$, we have
$$\sup_{-4r_i^2+s\le t\le s}f(t)\le \frac{12}{8}r_{i-1}^{-2}\le 2r_{i-1}^{-2}.$$
We can use induction just because $s-4r_i^{2}\ge 2r_i^2-\frac{1}{2}$ for any $s\ge -\frac{1}{2}+2r_{i-1}^2$. Thus we can choose $\ve=\ve(n, p_0)$ small enough such that
$$\sup_{-r_{i-1}^2+s\le t\le s}f(t)\le 2r_{i-1}^{-2}\le 2r_{i-1}^{-2}.$$
This proves the claim. By letting $s=0$ and scaling back to the original flow and repalcing $\ve$ by $\frac{1}{2}\ve$ finishes the proof of the Proposition.
\end{proof}

We can now apply Shi's estimate (cf \cite{Sh89}) to obtain curvature derivative estimate.

\begin{corollary}[Curvature Derivative Estimate]\label{Cor: curvature derivative estimate}
Given $n\geq 4$, $p_0>\frac{n}{2}$,  there exist constant $C(n,p_0)<\infty$ such that the following holds.

Let $(M^n,g(t)_{t\in(-1,1)})$ be a Ricci flow on a compact $n$-dimensional manifold $M$ with $\left\|R(t)\right\|_{L^{p_0}(M,g(t))}\leq 1$ for all $t\in(-1,1)$. If we have $\sup_{M}\left|\Rm(g(0))\right|\leq A$ with $A\geq 1$. Then we have
\begin{equation}\label{Equ: curvature derivative estimate}
\sup_{M}\left|\nabla\Rm(g(0))\right|\leq C(n, p_0)A^{\frac{3}{2}}.
\end{equation}
\end{corollary}
\begin{proof}
By our Global Backward Pseudolocality Proposition \ref{Prop: global backward pseudolocality}, there exists some constant $\ve=\ve(n, p_0)>0$, such that $\sup_{M}\left|\Rm(g(t))\right|\leq \ve^{-2}A$ for all $-\ve^2 A\leq t\leq 0$. Then Shi's estimate (cf \cite{Sh89}) gives us the desired curvature derivative estimate.
\end{proof}

\section{Proof of Theorem \ref{main theorem}}
\label{Proof of main Theorem}
In this section, we prove our main result Theorem \ref{main theorem}. First we recall the following result.
\begin{theorem}[\cite{CT06,Li10}]\label{Thm: curvature L2_small}
For any $\ve>0$ there exist constants $\delta(\ve)$ and $a(\ve)$ depending only on $\ve$ such that if $(M,g,x)$ is a complete $4$-dimensional manifold with $\sup_{B_r(x)}|\Ric|\le 3$ and $\int_{B_r(x)}|\Rm|^2\le \delta$ for some $r\le 1$, then  we have for some $s\ge a r$ that
\begin{align}
	s^4\fint_{B_s(x)}|\Rm|^2\le \ve.
\end{align}
\end{theorem}

Next we prove an $\ve$-regularity estimate of Ricci flow provided bounded Ricci curvature at time zero.

\begin{proposition}\label{Pro: eps_regularity_Ricci_flow_assume_bounded_Ricci}
Given $p_0>2$, there exist constants $\ve=\ve(p_0)>0$ and $C=C(p_0)<\infty$ such that the following holds.

Let $(M^4,g(t)_{t\in(-1,1)})$ be a Ricci flow on a compact $4$-dimensional manifold $M$ with $\left\|R(t)\right\|_{L^{p_0}(M,g(t))}\leq 1$ for all $t\in(-1,1)$. Assume at the time $t=0$ we have $\sup_{M}|\Ric(g(0))|\le 3$, $\int_{M}|\Rm(g(0))|^2dg(0)\le \ve$ and $\diam(M,g(0))\ge 1$, then
	\begin{align}
		\sup_{M}|\Rm(g(0))|\le C.
	\end{align}
\end{proposition}
\begin{proof}
By Theorem \ref{Thm: curvature L2_small} we have for any $\ve'$ there exist $\ve=\ve(\ve')$ and $\eta=\eta(\ve')>0$ such that if $\int_M|\Rm(g(0))|^2dg(0)\le \ve$, then for every $x\in (M,g(0))$ and some $1>s\ge \eta>0$ that
$$s^4\fint_{B(x,0,s)}|\Rm(g(0))|^2dg(0)\le \ve'.$$
This will be good enough to deduce the curvature estimates. Indeed, denote $|\Rm|(x_0)=\sup_{M}|\Rm(g(0))|=A$. We will show that $A\le s^{-2}$. Otherwise,  by Curvature Derivative Estimate Corollary \ref{Cor: curvature derivative estimate} we have $\sup_{M}|\nabla\Rm(g(0))|\le C_1A^{3/2}$ for some $C_1=C_1(p_0)$. Hence we have $|\Rm(g(0))|\geq \frac{1}{2}A$, on $B\left(x_0, 0, \frac{1}{2C_1}A^{-\frac{1}{2}}\right)$. Therefore, by the volume comparison we have for a constant $C_0=C_0(p_0)$ that
$$\frac{1}{4}\le A^{-2}\fint_{B\left(x_0, 0, \frac{1}{2C_1}A^{-\frac{1}{2}}\right)}|\Rm|^2dg(0)\le C_0s^4\fint_{B(x_0,0,s)}|\Rm|^2dg(0)\le C_0\ve',$$
which leads to a contradiction by choosing $\ve'=\ve'(p_0)$ small enough. Hence we finish the proof.
\end{proof}

Now we can show that the Ricci curvature is bounded if the $L^2$ curvature integral is small.

\begin{proposition}\label{Pro: Ricci estimate}
Given $p_0>2$, there exist constants $\ve=\ve(p_0)>0$ and $C=C(p_0)<\infty$ such that the following holds.

Let $(M^4,g(t)_{t\in(-1,1)})$ be a Ricci flow on a compact $4$-dimensional manifold $M$ with $\left\|R(t)\right\|_{L^{p_0}(M,g(t))}\leq 1$ for all $t\in(-1,1)$. Assume at the time $t=0$ we have $\diam(M,g(0))\ge 1$ and
$$\int_{M}|\Rm(g(0))|^2dg(0)\le \ve$$
then we have
\begin{equation}\label{Equ: Ricci estimate}
	    \sup_{M}|\Ric(g(0))|\le C.
\end{equation}
\end{proposition}
\begin{proof}
Let us fix $\ve$ as in Proposition \ref{Pro: eps_regularity_Ricci_flow_assume_bounded_Ricci} and denote the constant $C$ in Proposition \ref{Pro: eps_regularity_Ricci_flow_assume_bounded_Ricci} by $C_0$. Assume $\sup_{M}|\Ric(g(0))|=A>3$. We will show $A$ is bounded by a constant $C=C(n,p_0)$.

Consider the rescaling flow $\tilde{g}(t)=Ag(A^{-1}t)$ which satisfies the condition in Proposition \ref{Pro: eps_regularity_Ricci_flow_assume_bounded_Ricci} with scalar curvature integral bound $\left\|R(\ti{g}(t))\right\|_{L^{p_0}(M,\ti{g}(t))}\leq A^{-1+\frac{2}{p_0}}$ and $\sup_{M}|\Ric(\tilde{g}(0))|=1$. Applying Proposition \ref{Pro: eps_regularity_Ricci_flow_assume_bounded_Ricci}, we have
$$\sup_{M}|\Rm(\tilde{g}(0))|\le C_0.$$
Hence we can apply Global Backward Pseudolocality Proposition \ref{Prop: global backward pseudolocality} to find some $\ve_0=\ve_0(p_0)>0$ such that
$$\sup_{-\ve_0^2 C_0^{-1}\leq t\leq 0}\sup_{M}|\Rm(\tilde{g}(t))|\le \ve_0^{-2}C_0.$$
Then we can apply the Improved Curvature Estimate Proposition \ref{Prop: improved curvature estimate} to find that for some $C_1=C_1(p_0)$
$$1=\sup_{M}|\Ric(\tilde{g}(0))|^2\le C_1\cdot A^{-1+\frac{2}{p_0}}\cdot \ve_0^{-1-\frac{2}{p_0}}C_0^{\frac{1}{2}-\frac{1}{p_0}}.$$
But we have $p_0>2$, hence we obtain
$$A\leq C(p_0).$$
This finishes the proof.
\end{proof}

Now we can prove our main Theorem \ref{main theorem}.
\begin{proof}[Proof of Theorem \ref{main theorem}]
Theorem \ref{main theorem} follows directly from Proposition \ref{Pro: eps_regularity_Ricci_flow_assume_bounded_Ricci} and Proposition \ref{Pro: Ricci estimate}.
\end{proof}

\end{document}